\documentclass[12pt]{amsart}
\usepackage{cases}
\usepackage{cases}
\usepackage{cases}
\usepackage{mathrsfs}
\usepackage{amsfonts}
\usepackage{amscd}
\usepackage{latexsym,amssymb}
\usepackage{a4wide}
\usepackage{epic,eepic,latexsym, amssymb, amscd, amsfonts, xypic, floatflt}
\usepackage{amsthm}
\usepackage{amsmath}
\usepackage{amscd}
\usepackage{graphicx}
\usepackage[mathscr]{eucal}
\input xy
\xyoption{all}
\usepackage{verbatim}

\numberwithin{equation}{section}
\begin{document}
\theoremstyle{plain}
\newtheorem{thm}{Theorem}[section]
\newtheorem{theorem}[thm]{Theorem}
\newtheorem{lemma}[thm]{Lemma}
\newtheorem{corollary}[thm]{Corollary}
\newtheorem{proposition}[thm]{Proposition}
\newtheorem{conjecture}[thm]{Conjecture}
%%%%%%%%%%%%%%%%%%%% Text roman %%%%%%%%%%%%%%%%%%%%%%%%%%%%%
\theoremstyle{definition}
\newtheorem{remark}[thm]{Remark}
\newtheorem{remarks}[thm]{Remarks}
\newtheorem{definition}[thm]{Definition}
\newtheorem{example}[thm]{Example}

\newcommand{\caA}{{\mathcal A}}
\newcommand{\caB}{{\mathcal B}}
\newcommand{\caC}{{\mathcal C}}
\newcommand{\caD}{{\mathcal D}}
\newcommand{\caE}{{\mathcal E}}
\newcommand{\caF}{{\mathcal F}}
\newcommand{\caG}{{\mathcal G}}
\newcommand{\caH}{{\mathcal H}}
\newcommand{\caI}{{\mathcal I}}
\newcommand{\caJ}{{\mathcal J}}
\newcommand{\caK}{{\mathcal K}}
\newcommand{\caL}{{\mathcal L}}
\newcommand{\caM}{{\mathcal M}}
\newcommand{\caN}{{\mathcal N}}
\newcommand{\caO}{{\mathcal O}}
\newcommand{\caP}{{\mathcal P}}
\newcommand{\caQ}{{\mathcal Q}}
\newcommand{\caR}{{\mathcal R}}
\newcommand{\caS}{{\mathcal S}}
\newcommand{\caT}{{\mathcal T}}
\newcommand{\caU}{{\mathcal U}}
\newcommand{\caV}{{\mathcal V}}
\newcommand{\caW}{{\mathcal W}}
\newcommand{\caX}{{\mathcal X}}
\newcommand{\caY}{{\mathcal Y}}
\newcommand{\caZ}{{\mathcal Z}}
%mathfrak
\newcommand{\fA}{{\mathfrak A}}
\newcommand{\fB}{{\mathfrak B}}
\newcommand{\fC}{{\mathfrak C}}
\newcommand{\fD}{{\mathfrak D}}
\newcommand{\fE}{{\mathfrak E}}
\newcommand{\fF}{{\mathfrak F}}
\newcommand{\fG}{{\mathfrak G}}
\newcommand{\fH}{{\mathfrak H}}
\newcommand{\fI}{{\mathfrak I}}
\newcommand{\fJ}{{\mathfrak J}}
\newcommand{\fK}{{\mathfrak K}}
\newcommand{\fL}{{\mathfrak L}}
\newcommand{\fM}{{\mathfrak M}}
\newcommand{\fN}{{\mathfrak N}}
\newcommand{\fO}{{\mathfrak O}}
\newcommand{\fP}{{\mathfrak P}}
\newcommand{\fQ}{{\mathfrak Q}}
\newcommand{\fR}{{\mathfrak R}}
\newcommand{\fS}{{\mathfrak S}}
\newcommand{\fT}{{\mathfrak T}}
\newcommand{\fU}{{\mathfrak U}}
\newcommand{\fV}{{\mathfrak V}}
\newcommand{\fW}{{\mathfrak W}}
\newcommand{\fX}{{\mathfrak X}}
\newcommand{\fY}{{\mathfrak Y}}
\newcommand{\fZ}{{\mathfrak Z}}

\newcommand{\bA}{{\mathbb A}}
\newcommand{\bB}{{\mathbb B}}
\newcommand{\bC}{{\mathbb C}}
\newcommand{\bD}{{\mathbb D}}
\newcommand{\bE}{{\mathbb E}}
\newcommand{\bF}{{\mathbb F}}
\newcommand{\bG}{{\mathbb G}}
\newcommand{\bH}{{\mathbb H}}
\newcommand{\bI}{{\mathbb I}}
\newcommand{\bJ}{{\mathbb J}}
\newcommand{\bK}{{\mathbb K}}
\newcommand{\bL}{{\mathbb L}}
\newcommand{\bM}{{\mathbb M}}
\newcommand{\bN}{{\mathbb N}}
\newcommand{\bO}{{\mathbb O}}
\newcommand{\bP}{{\mathbb P}}
\newcommand{\bQ}{{\mathbb Q}}
\newcommand{\bR}{{\mathbb R}}
\newcommand{\bT}{{\mathbb T}}
\newcommand{\bU}{{\mathbb U}}
\newcommand{\bV}{{\mathbb V}}
\newcommand{\bW}{{\mathbb W}}
\newcommand{\bX}{{\mathbb X}}
\newcommand{\bY}{{\mathbb Y}}
\newcommand{\bZ}{{\mathbb Z}}
\newcommand{\id}{{\rm id}}
%%%%%%%%%%%%%%%%%%%%%%%%%%%%%%%%%%%%%%%%%%%%%%%%%%%%%%%%%%%%%%

\title[A new approach to Lickorish-Millett type formulae]
{A new approach to Lickorish-Millett type formulae}
\author[Xin Zhou, Shengmao Zhu]{Xin Zhou, Shengmao Zhu}
\address{Center of Mathematical Sciences \\Zhejiang University \\Hangzhou, 310027, China }
\email{risingsun.up@gmail.com, shengmaozhu@126.com}

%%%%%%%%%%%%%%%%%%%%%%%%%%%%%%%%%%%%%%%%%%%%%%%%%%%%%%%%%%%%%%%%%%%%%%%%%%%%%%%%%%%%%%%%%%%%%%%%
%                                          Abstract                                            %
%%%%%%%%%%%%%%%%%%%%%%%%%%%%%%%%%%%%%%%%%%%%%%%%%%%%%%%%%%%%%%%%%%%%%%%%%%%%%%%%%%%%%%%%%%%%%%%%
\begin{abstract}
In this paper, we introduce a new method to prove the
Lickorish-Millett type formulae for colored HOMFLY-PT polynomials of
links.
%\\ \noindent{Keywords}: Colored HOMFLY-PT invariants, Lickorish-Millett type formulae, LMOV conjecture.
%\\ \noindent{MSC classes}:  57M25, 57M27
\end{abstract}

\maketitle

%\setcounter{tocdepth}{5} \setcounter{page}{1}

%\tableofcontents
%\newpage

\section{Introduction}
The HOMFLY-PT polynomial is a two variables link invariant that was
first discovered by Freyd-Yetter, Lickorish-Millett, Ocneanu, Hoste
and Przytycki-Traczyk. Given an oriented link $\mathcal{L}$ in
$S^3$, its HOMFLY-PT polynomial $P(\mathcal{L},q,t)$ satisfies the
following skein relation,
\begin{align}
tP(\mathcal{L}_+;
q,t)-t^{-1}P(\mathcal{L}_-;q,t)=(q-q^{-1})P(\mathcal{L}_0; q,t)
\end{align}
where we will use the notation $(\mathcal{L}_+,
\mathcal{L}_-,\mathcal{L}_0)$ to denote the Conway triple throughout
this paper. Given an initial value $P(U; q,t)=1$ for an unknot $U$,
one can compute the HOMFLY-PT polynomial for a given oriented link
recursively through the above formula (1.1). We can give the
definition of the HOMFLY-PT polynomial through the HOMFLY-PT skein
of the plane $\mathcal{S}(\mathbb{R}^2)$. Any link diagram of a link
$\mathcal{L}$ in $\mathcal{S}(\mathbb{R}^2)$ will be equal to a
scalar denoted by $\mathcal{H}(\mathcal{L};q,t)$. In particular, for
an unknot $U$, $\mathcal{H}(U;q,t)=\frac{t-t^{-1}}{q-q^{-1}}$.  Then
the HOMFLY-PT polynomial of $\mathcal{L}$ is defined by
\begin{align}
P(\mathcal{L};q,t)=\frac{t^{-w(\mathcal{L})}\mathcal{H}(\mathcal{L};q,t)}{\mathcal{H}(U;q,t)},
\end{align}
where $w(\mathcal{L})$ is the writhe number of $\mathcal{L}$. In
Section 2, by a simple observation, we obtain the following
structural theorem for HOMFLY-PT polynomials first showed by
Lickorish-Millett in \cite{LM}.
\begin{proposition}
For a link $\mathcal{L}$ with $L$ components, we have the following
expansions:
\begin{align}
\mathcal{H}(\mathcal{L};q,t)=\sum_{g\geq 0}h_{2g-L}^{\mathcal{L}}(t)z^{2g-L},\\
P(\mathcal{L};q,t)=\sum_{g\geq
0}p_{2g+1-L}^{\mathcal{L}}(t)z^{2g+1-L},
\end{align}
where $z=q-q^{-1}$. The polynomials $h_{2g-L}^{\mathcal{L}}(t)$ and
$p_{2g+1-L}^{\mathcal{L}}(t)$ will be called coefficient
polynomials. According to the formula (1.2), it clear that we have
the relations
\begin{align}
h_{2g-L}^{\mathcal{L}}(t)=p_{2g+1-L}^{\mathcal{L}}(t)t^{w(\mathcal{L})}(t-t^{-1}),
\end{align}
for $g\geq 0$.
\end{proposition}
Furthermore, Lickorish-Millett \cite{LM}  first studied the explicit
expressions of the coefficient polynomials
$p_{2g+1-L}^{\mathcal{L}}(t)$ appearing in the above expansions
(1.4). They studied the first coefficient $p_{1-L}^{\mathcal{L}}(t)$
for a link $\mathcal{L}$ with $L$ components and obtained the
following Theorem 1.4. Then in \cite{KM}, Kanenobu-Miyazawa
carefully computed the second and third coefficients
$p_{3-L}^{\mathcal{L}}(t)$, $p_{5-L}^{\mathcal{L}}(t)$ by using the
similar method as in \cite{LM}. All these formulas about the
coefficient polynomials $p_{2g+1-L}^{\mathcal{L}}(t)$ will be called
the "Lickorish-Millett" type formulas. In this paper, we introduce a
completely new method to study these coefficient polynomials
$p_{2g+1-L}^{\mathcal{L}}(t)$. This method is motivated by the work
of \cite{LP} in the proof of Labastida-Mari\~no-Ooguri-Vafa (LMOV)
conjecture.

Let us fix some notation. For a link $\mathcal{L}$ with $L$
components, let $\mathbf{I}$ be a subset of
$\mathbf{L}=\{1,2,...,L\}$. Denote by $\mathcal{L}_{\mathbf{I}}$ the
sub-link obtained by removing the components whose sub-indices are
not contained in $\mathbf{I}$. For example, if $\mathcal{L}$ is a
link of two components $\mathcal{K}_1$ and $\mathcal{K}_2$, then
$\mathcal{L}_{\{2\}}=\mathcal{K}_2$. For $L\geq 1$ and $l\leq L$, we
will use the notation $\cup_{s=1}^{l}\mathbf{I}_{s}=\mathbf{L}$ to
denote a nonempty disjoint ordered decomposition of the set
$\mathbf{L}=\{1,..,L\}$, i.e. every $\mathbf{I}_s\subset
\mathbf{L}$, $\mathbf{I}_s \neq \emptyset$ and $\mathbf{I}_s\cap
\mathbf{I}_t=\emptyset$ for $s\neq t$ such that $\mathbf{I}_{1}\cup
\mathbf{I}_2\cup \cdots \cup \mathbf{I}_l=\mathbf{L}$, and different
orders of $\mathbf{I}_{1}, \mathbf{I}_2, \cdots ,\mathbf{I}_l$ give
different decompositions.

 We introduce the intermediate
invariant $\mathcal{F}(\mathcal{L};q,t)$  as follow:
\begin{align}
\mathcal{F}(\mathcal{L};q,t)=\sum_{l=1}^L\frac{(-1)^{l-1}}{l}
\sum_{\cup_{s=1}^l\mathbf{I}_s=\mathbf{L}}
\prod_{s=1}^l\mathcal{H}(\mathcal{L}_{\mathbf{I}_s};q,t),
\end{align}
where the summation $\sum_{\cup_{s=1}^l\mathbf{I}_s=\mathbf{L}}$
denotes the sum over all nonempty disjoint ordered decompositions
$\cup_{s=1}^l\mathbf{I}_s=\mathbf{L}$ throughout this paper.

In Section 3, we will prove that
\begin{proposition}
\begin{align}
\deg_{z} \mathcal{F}(\mathcal{L};q,t)\geq L-2.
\end{align}
where we use the notation $\deg_z f$ to denote the lowest degree of
$z$ in a polynomial $f\in z^{-L}\mathbb{Q}[z^2,t^{\pm 1}]$.
\end{proposition}

Combining the expansion formula (1.3) for
$\mathcal{H}(\mathcal{L};q,t)$ and Proposition 1.2, we immediately
have
\begin{theorem}
For a link $\mathcal{L}$ with $L$ components, the coefficient
polynomials $h_{2g-L}^{\mathcal{L}}(t)$ for $g=0,..,L-2$ can be
expressed as follow:
\begin{align}
h_{2g-L}^{\mathcal{L}}(t)=\sum_{l=2}^{L}\frac{(-1)^{l}}{l}\sum_{
\sum_{s=1}^lg_s=g}\sum_{\cup_{s=1}^l
\mathbf{I}_s=\mathbf{L}}\prod_{s=1}^l
h_{2g_s-I_s}^{\mathcal{L}_{\mathbf{I}_s}}(t).
\end{align}
where the second summation $\sum_{\sum_{s=1}^lg_s=g}$ denotes the
sum over all nonnegative integers $g_1,g_2...,g_l$ such that
$\sum_{s=1}^lg_s=g$.
\end{theorem}
In fact, by using the relation (1.5), Theorem 1.3 directly gives us
$L-1$ Lickorish-Millett type formulas for the coefficient
polynomials $p_{2g+1-L}^{\mathcal{L}}(t)$, $g=0,1,...,L-2$.

For examples, when $g=0$, we reprove the following result due to
Lickorish-Millett ( See Proposition 22 in \cite{LM}).
\begin{theorem}
The first coefficient $p_{1-L}^{\caL}(t)$ in the expansion
\begin{align}
P(\mathcal{L};q,t)=\sum_{g\geq 0}p_{2g+1-L}^{\caL}(t)z^{2g+1-L}
\end{align}
is given by
\begin{align}
p_{1-L}^{\caL}(t)=t^{-2lk(\caL)}(t-t^{-1})^{L-1}\prod_{\alpha=1}^{L}p_{0}^{\caK_{\alpha}}(t),
\end{align}
where $lk(\caL)$ is the linking number of $\mathcal{L}$.
\end{theorem}

When $g=1$,  we recover  the following theorem of Kanenobu-Miyazawa
(See Theorem 1.1 in \cite{KM}).
\begin{theorem}
The second coefficient $p_{3-L}^{\caL}(t)$ in the HOMFLY-PT
polynomial
\begin{align}
P(\mathcal{L};q,t)=\sum_{g\geq 0}p_{2g+1-L}^{\caL}(t)z^{2g+1-L}
\end{align}
is given by
\begin{align}
p_{3-L}^{\caL}(t)=&t^{-2lk(\caL)}\left((t-t^{-1})^{L-2}\sum_{1\leq
\beta<\gamma \leq L}
t^{2lk(\caL_{(\alpha\beta)})}p_{1}^{\caL_{(\alpha\beta)}}(t)\prod_{\alpha\neq
\beta,\gamma}p_{0}^{\caK_{\alpha}}(t)\right.\\\nonumber
&\left.-(L-2)(t-t^{-1})^{L-1}\sum_{\beta=1}^Lp_{2}^{\caK_{\alpha}}(t)\prod_{\alpha\neq
\beta}p_{0}^{\caK_{\alpha}}(t)\right).
\end{align}
\end{theorem}

By using Theorem 1.3,  we reduce the proofs of Theorem 1.4 and
Theorem 1.5 to some combinatorial identities which will be shown in
Section 5.

Finally, in the appendix, for the reader's convenience, we describe
the classical approach to Theorem 1.4 and Theorem 1.5 as shown in
\cite{LM} and \cite{KM} with a slightly different method.

{\bf Acknowledgements.} This paper grows out of an old note by the
second author when he learned of the Labastida-Mari\~no-Ooguri-Vafa
question from Prof. Kefeng Liu. He would like to thank Prof. Kefeng
Liu for sharing with him many  new ideas. He also appreciates the
collaboration with Qingtao Chen, Kefeng Liu and Pan Peng in this
area and many valuable discussions with them within the past years.
The authors are grateful to the referees for their valuable comments
and suggestions which greatly improved the presentation of the
content.

The research of S. Zhu is supported by the National Science
Foundation of China grant No. 11201417 and the China Postdoctoral
Science special Foundation No. 2013T60583.

\section{HOMFLY-PT polynomials}
The HOMFLY-PT polynomial of an oriented link $\mathcal{L}$ can be
defined by using the framed HOMFLY-PT skein of the plane
$\mathcal{S}(\mathbb{R}^2)$ which is the set of linear combinations
of oriented link diagrams, modulo two relations given in Figure 1,
where $z=q-q^{-1}$.
\begin{figure}[!htb]
\begin{center}
\includegraphics[width=150 pt]{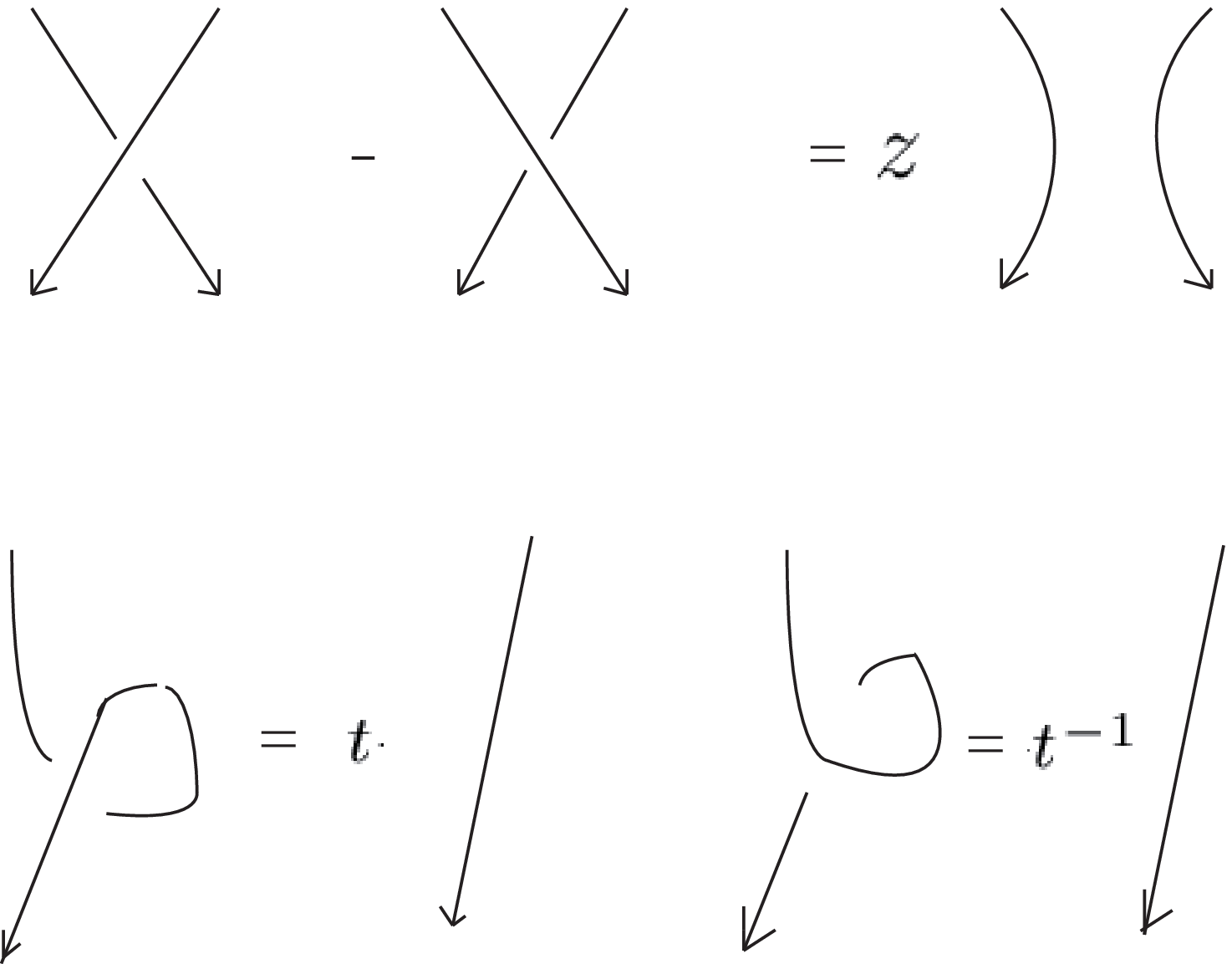}
\caption{}
\end{center}
\end{figure}
It is easy to follow that the removal an unknot requires  a
multiplication by the scalar $s=\frac{t-t^{-1}}{q-q^{-1}}$, i.e we
have the relation
 showed in Figure 2.
\begin{figure}[!htb]
\begin{center}
\includegraphics[width=80 pt]{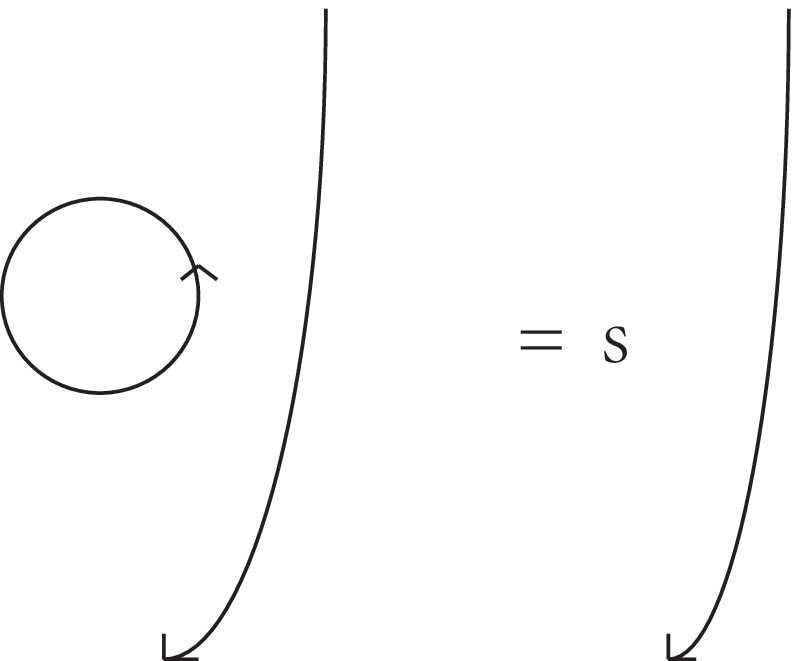}
\caption{}
\end{center}
\end{figure}

The planar projection of an oriented link, $\mathcal{L}$, gives an
oriented diagram that is denoted by $D_{\mathcal{L}}$. By using the
above two relations, any diagram $D_\mathcal{L}$ is equal to a
scalar. We denote the resulting scalar by $\langle D_{\mathcal{L}}
\rangle \in \Lambda$. Then the (framed unreduced) HOMFLY-PT
polynomial $\mathcal{H}(\mathcal{L};q,t)$ of the link $\mathcal{L}$
is defined by $\mathcal{H}(\mathcal{L};q,t)=\langle
D_\mathcal{L}\rangle$. We use the convention  $\langle \ \rangle=1$
for the empty diagram, for the unknot $U$,
$\mathcal{H}(U;q,t)=\frac{t-t^{-1}}{q-q^{-1}}$. The two relations
shown in Figure 1 lead to
\begin{align}
\mathcal{H}(\mathcal{L}_+;q,t)-\mathcal{H}(\mathcal{L}_-;q,t)=z\mathcal{H}(\mathcal{L}_0;q,t),\\
\mathcal{H}(\mathcal{L}^{+1};q,t)=t\mathcal{H}(\mathcal{L};q,t),\
\mathcal{H}(\mathcal{L}^{-1};q,t)=t^{-1}\mathcal{H}(\mathcal{L};q,t).
\end{align}
The classical HOMFLY-PT polynomial of a link $\mathcal{L}$ is given
by
\begin{align}
P(\mathcal{L};q,t)=\frac{t^{-w(\mathcal{L})}\mathcal{H}(\mathcal{L};q,t)}{\mathcal{H}(U;q,t)}.
\end{align}

By the observation used in \cite{CLPZ}, for a link $\mathcal{L}$
with $L$ components, we introduce the link polynomial
$\check{\mathcal{H}}(\mathcal{L};q,t)$ defined as
\begin{align}
\check{\mathcal{H}}(\mathcal{L};q,t)=z^{L}\mathcal{H}(\mathcal{L};q,t).
\end{align}

By comparing the numbers of components of links $\mathcal{L}_+,
\mathcal{L}_-, \mathcal{L}_0$ in Conway triple,  the relations (2.1)
and (2.2) of $\mathcal{H}(\mathcal{L};q,t)$ give us
\begin{align}
\check{\mathcal{H}}(\mathcal{L}_+;q,t)-\check{\mathcal{H}}(\mathcal{L}_-;q,t)=z^{2\epsilon}
\check{\mathcal{H}}(\mathcal{L}_0;q,t)\\
\check{\mathcal{H}}(\mathcal{L}^{+1};q,t)=t\check{\mathcal{H}}(\mathcal{L};q,t),
\
\check{\mathcal{H}}(\mathcal{L}^{-1};q,t)=t^{-1}\check{\mathcal{H}}(\mathcal{L};q,t)
\end{align}
where $\epsilon=0$ if the crossing is a self-crossing of a knot,
$\epsilon=1$ if this crossing is a crossing between two components
of a link. Let $U$ be the unknot, we have
$\check{\mathcal{H}}(U;q,t)=t-t^{-1}$. By using the relations (2.5)
and (2.6) recursively, we have
\begin{lemma}
For any link $\mathcal{L}$, $\check{\mathcal{H}}(\mathcal{L};q,t)\in
\mathbb{Z}[z^2,t^{\pm 1}]$.
\end{lemma}
In other words,  for a link $\mathcal{L}$, there are polynomials of
$t$ denoted by $\check{h}_{g}^{\mathcal{L}}(t)\in \mathbb{Z}[t^{\pm
1}]$, such that $\check{\mathcal{H}}(\mathcal{L};q,t)$ has the
following expansion:
\begin{align}
 \check{\mathcal{H}}(\mathcal{L};q,t)=\sum_{g\geq 0}\check{h}_{g}^{\mathcal{L}}(t)z^{2g}.
\end{align}
It immediately implies the following structural formulae for
$\mathcal{H}(\mathcal{L};q,t)$ and $P(\mathcal{L};q,t)$.
\begin{proposition}
For a link $\mathcal{L}$ with $L$ components
\begin{align}
\mathcal{H}(\mathcal{L};q,t)=\sum_{g\geq 0}h_{2g-L}^{\mathcal{L}}(t)z^{2g-L},\\
P(\mathcal{L};q,t)=\sum_{g\geq
0}p_{2g+1-L}^{\mathcal{L}}(t)z^{2g+1-L},
\end{align}
where $h_{2g-L}^{\mathcal{L}}(t)=\check{h}_{g}^{\mathcal{L}}(t)$,
and
$p_{2g+1-L}^{\mathcal{L}}(t)=\frac{t^{-w(\mathcal{L})}h_{2g-L}^{\mathcal{L}}(t)}{t-t^{-1}}$
\end{proposition}
We remark that the formula (2.9) was first showed in the paper
\cite{LM}.

\section{An intermediate invariant}
In this section,  we introduce an intermediate invariant whose
properties imply the Lickorish-Millett type formulae.  With the same
notations shown in Section 1, the intermediate invariant
$\mathcal{F}(\mathcal{L};q,t)$ is given by the formula (1.6)
\begin{align}
\mathcal{F}(\mathcal{L};q,t)=\sum_{l=1}^L\frac{(-1)^{l-1}}{l}
\sum_{\cup_{s=1}^l\mathbf{I}_s=\mathbf{L}}
\prod_{s=1}^l\mathcal{H}(\mathcal{L}_{\mathbf{I}_s};q,t).
\end{align}

By formula (2.8) in Proposition 2.2, it is obvious that
\begin{align}
\mathcal{F}(\mathcal{L};q,t)\in z^{-L}\mathbb{Q}[z^2,t^{\pm 1}].
\end{align}
In fact, we have a more precise structure for
$\mathcal{F}(\mathcal{L};q,t)$.
\begin{proposition}
If we use the notation $\deg_z f$ to denote the lowest degree of $z$
in a polynomial $f\in z^{-L}\mathbb{Q}[z^2,t^{\pm 1}]$.
\begin{align}
\deg_{z} \mathcal{F}(\mathcal{L};q,t)\geq L-2.
\end{align}
\end{proposition}
\begin{proof}
For a link $\mathcal{L}$ with $L$ components, $\mathcal{K}_\alpha,
\alpha=1,..,L$. Let us consider a crossing which is a crossing
between two different components $\mathcal{K}_\alpha$ and
$\mathcal{K}_\beta$ of $\mathcal{L}$, we have
\begin{align}
&\mathcal{F}(\mathcal{L}_+;q,t)-\mathcal{F}(\mathcal{L}_-;q,t)\\\nonumber
&=\sum_{l=1}^{L}\frac{(-1)^{l-1}}{l}\sum_{\cup_{s=1}^l \mathbf{I}_s=
\mathbf{L}
}\left(\prod_{s=1}^l\mathcal{H}((\mathcal{L}_+)_{\mathbf{I}_s};q,t)-\prod_{s=1}^l\mathcal{H}((\mathcal{L}_-)_{\mathbf{I}_s};q,t)\right)
\end{align}
Note that on the right hand side of the above formula, only the
terms that contain the crossings in the same sub-link
$\mathcal{L}_{\mathbf{I}_s}$ survive. By skein relation (2.1),  one
has
\begin{align}
\mathcal{F}(\mathcal{L}_+;q,t)-\mathcal{F}(\mathcal{L}_-;q,t)
=z\mathcal{F}(\mathcal{L}_0;q,t).
\end{align}
For brevity,  we introduce the following notation for a link
$\mathcal{L}$ with $L$ components
\begin{align}
\hat{\mathcal{F}}(\mathcal{L};q,t)=\frac{\mathcal{F}(\mathcal{L};q,t)}{z^{L}}.
\end{align}
So we only need to prove
\begin{align}
\deg_{z}\hat{\mathcal{F}}(\mathcal{L};q,t)\geq -2.
\end{align}
We see that the formula (3.5) becomes
\begin{align}
\hat{\mathcal{F}}(\mathcal{L}_+;q,t)-\hat{\mathcal{F}}(\mathcal{L}_-;q,t)
=\hat{\mathcal{F}}(\mathcal{L}_0;q,t).
\end{align}
Then the proof of the formula (3.7) will be finished by induction on
the number of components of link.

For $L=1$, let us consider a knot, $\mathcal{K}$. By Proposition
2.2,
\begin{align}
\hat{\mathcal{F}}(\mathcal{K};q,t)=\frac{\mathcal{F}(\mathcal{K};q,t)}{z}=\frac{\mathcal{H}(\mathcal{K};q,t)}{z}\in
z^{-2}\mathbb{Z}[z^{2},t^{\pm 1}].
\end{align}
So $\deg_{z}\hat{\mathcal{F}}(\mathcal{K};q,t)\geq -2$ holds.

Now we assume the formula (3.7) holds for any link with number of
components $\leq L-1$. Let us consider a link $\mathcal{L}$ with $L$
components, $\mathcal{K}_\alpha, \alpha=1,...,L$. If we assume the
formula (3.7) does not holds for $\mathcal{L}$, i.e
$\deg_{z}\hat{\mathcal{F}}(\mathcal{L};q,t)< -2$, then we can find
the contradiction.

We use the notation $n^{+}_{\alpha,\beta}$ to denote the number of
positive crossings between two components $\mathcal{K}_\alpha$ and
$\mathcal{K}_\beta$. Without loss of the generality, we assume
$n^{+}_{\alpha,\beta}>0$. So we can apply the relation (3.8) to such
a positive crossing in $\mathcal{L}$.
\begin{align}
\hat{\mathcal{F}}(\mathcal{L};q,t)-\hat{\mathcal{F}}(\mathcal{L}_-;q,t)
=\hat{\mathcal{F}}(\mathcal{L}_0;q,t).
\end{align}
Since $\deg_{z}\hat{\mathcal{F}}(\mathcal{L}_0;q,t)\geq -2$ by the
induction hypothesis, we must have
\begin{align}
\deg_{z}\hat{\mathcal{F}}(\mathcal{L}_-;q,t)=\deg_z\hat{\mathcal{F}}(\mathcal{L};q,t)<-2.
\end{align}
We can also apply the relation (3.8) to the link $\mathcal{L}_-$
recursively until the components $\mathcal{K}_\alpha$ and
$\mathcal{K}_\beta$ are separated. In this way, we finally obtain a
separated link
$\tilde{\mathcal{L}}=\otimes_{\alpha=1}^L\mathcal{K}_\alpha$, i.e.
$\mathcal{K}_\alpha$ and $\mathcal{K}_\beta$ are separated for
arbitrary $\alpha\neq\beta$. Then
\begin{align}
\deg_{z}\hat{\mathcal{F}}(\tilde{\mathcal{L}};q,t)=\deg_z\hat{\mathcal{F}}(\mathcal{L};q,t)<-2.
\end{align}

However, for the separated link
$\otimes_{\alpha=1}^L\mathcal{K}_\alpha$, we always have
$\mathcal{H}(\otimes_{\alpha=1}^L\mathcal{K}_\alpha;q,t)=\prod_{\alpha=1}^L\mathcal{H}(\mathcal{K}_\alpha;q,t)$.
Thus
\begin{align}
\mathcal{F}(\tilde{\mathcal{L}};q,t)&=\sum_{l=1}^L\frac{(-1)^{l-1}}{l}
\sum_{\cup_{s=1}^l\mathbf{I}_s=\mathbf{L}}
\prod_{s=1}^l\mathcal{H}(\mathcal{L}_{\mathbf{I}_s};q,t)\\\nonumber
&=\left(\sum_{l=1}^L\frac{(-1)^{l-1}}{l}
\sum_{\cup_{s=1}^l\mathbf{I}_s=\mathbf{L}}
1\right)\prod_{\alpha=1}^L\mathcal{H}(\mathcal{K}_\alpha;q,t)\\\nonumber
&=0,
\end{align}
by the combinatorial identity (5.2) in Section 5. This contradicts
the statement $\deg_z\hat{\mathcal{F}}(\tilde{\mathcal{L}};q,t)<-2$.
So the proof of Proposition 3.1 is completed.
\end{proof}

We remark that Proposition 3.1 was first proved by Liu-Peng
\cite{LP} in order to prove the Labastida-Mari\~no-Ooguri-Vafa
conjecture \cite{LMV,OV}.

\section{Proofs of the Lickorish-Millett type formulae}
By  Proposition  2.2, for every sublink $\mathcal{L}_{\mathbf{I}_s}$
of $\mathcal{L}$, let $I_s=|\mathbf{I}_s|$, the number of the
elements in the set $\mathbf{I}_s$.  we have the following
expansion:
\begin{align}
\mathcal{H}(\mathcal{L}_{\mathbf{I}_s};q,t)=\sum_{g_s\geq
0}h_{2g_s-I_s}^{\mathcal{L}_{\mathbf{I}_s}}(t)z^{2g_s-I_s}.
\end{align}
Therefore, substituting them to the formula (3.1)
\begin{align}
\mathcal{F}(\mathcal{L};q,t)=\sum_{g\geq
0}\left(\sum_{l=1}^{L}\frac{(-1)^{l-1}}{l}\sum_{\sum_{s=1}^lg_s=g}\sum_{\cup_{s=1}^l
\mathbf{I}_s=\mathbf{L} }\prod_{s=1}^l
h_{2g_s-I_s}^{\mathcal{L}_{\mathbf{I}_s}}(t)\right)z^{2g-L}.
\end{align}
For brevity, we let
\begin{align}
f_{2g-L}(t)=\sum_{l=1}^{L}\frac{(-1)^{l-1}}{l}\sum_{\sum_{s=1}^lg_s=g}\sum_{\cup_{s=1}^l
\mathbf{I}_s=\mathbf{L} }\prod_{s=1}^l
h_{2g_s-I_s}^{\mathcal{L}_{\mathbf{I}_s}}(t).
\end{align}
Then Proposition 3.1 tells us
\begin{align}
f_{-L}(t)=f_{2-L}(t)=\cdots =f_{L-4}(t)=0.
\end{align}
These $L-1$ equations will give rise to $L-1$ relations of the
coefficient polynomials.
\begin{theorem}
For a link $\mathcal{L}$ with $L$ components, the coefficients
polynomials $h_{2g-L}^{\mathcal{L}}(t)$, for $g=0,..,L-2$, can be
expressed as follow:
\begin{align}
h_{2g-L}^{\mathcal{L}}(t)=\sum_{l=2}^{L}\frac{(-1)^{l}}{l}\sum_{\sum_{s=1}^lg_s=g}\sum_{\cup_{s=1}^l
\mathbf{I}_s=\mathbf{L} }\prod_{s=1}^l
h_{2g_s-I_s}^{\mathcal{L}_{\mathbf{I}_s}}(t).
\end{align}
\end{theorem}

In the following, we will carefully study the first two equations
$f_{L}(t)=0$ and $f_{2-L}(t)=0$ and recover some classical results
of the Lickorish-Millett type formulae in \cite{LM,KM}.

\subsection{Case 1:  $f_{-L}(t)=0$}
By formula (4.3),
\begin{align}
f_{-L}(t)=\sum_{l=1}^{l}\frac{(-1)^{l-1}}{l}\sum_{\cup_{s=1}^l
\mathbf{I}_s=\mathbf{L}}\prod_{s=1}^lh_{-I_s}^{\mathcal{L}_{\mathbf{I}_s}}(t)=0.
\end{align}
Thus, we have
\begin{align}
h_{-L}^{\mathcal{L}}(t)=\sum_{l=2}^{l}\frac{(-1)^{l}}{l}\sum_{\cup_{s=1}^l
\mathbf{I}_s=\mathbf{L}}\prod_{s=1}^lh_{-I_s}^{\mathcal{L}_{\mathbf{I}_s}}(t).
\end{align}

For $L=2$, it is clear that
\begin{align}
h_{-2}^{\mathcal{L}}(t)=h_{-1}^{\mathcal{K}_{1}}(t)h_{-1}^{\mathcal{K}_{2}}(t).
\end{align}
So we guess that  $h_{-L}^{\mathcal{L}}(t)$ for a general $L$ takes
the form
$h_{-L}^{\mathcal{L}}(t)=\prod_{\alpha=1}^{L}h_{-1}^{\mathcal{K}_{\alpha}}(t)$.
We prove it by induction. We assume it holds for the number of
components of link $\leq L-1$. For $L$,  we have
\begin{align}
h_{-L}^{\mathcal{L}}(t)&=\sum_{l=2}^{L}\frac{(-1)^l}{l}\sum_{\cup_{s=1}^l\mathbf{I}_s=\mathbf{L}}\prod_{\alpha=1}^{L}
h_{-1}^{\mathcal{K}_{\alpha}}(t)\\\nonumber &=\prod_{\alpha=1}^{L}
h_{-1}^{\mathcal{K}_{\alpha}}(t)\left(\sum_{l=2}^{L}\frac{(-1)^l}{l}\sum_{\cup_{s=1}^l\mathbf{I}_s=\mathbf{L}}1\right)\
\text{By the identity (5.2) }\\\nonumber &=\prod_{\alpha=1}^{L}
h_{-1}^{\mathcal{L}_{\alpha}}(t)
\end{align}

By using the formula (1.5),
$h_{-L}^{\mathcal{L}}(t)=p_{1-L}^{\mathcal{L}}(t)t^{w(\mathcal{L})}(t-t^{-1})$,
the formula (4.9) changes to
\begin{align}
p_{1-L}^{\mathcal{L}}(t)=t^{-2lk(\mathcal{L})}(t-t^{-1})^{L-1}\prod_{\alpha=1}^Lp_{0}^{\mathcal{K}_\alpha}(t),
\end{align}
where the identity
$w(\mathcal{L})=\sum_{\alpha=1}^Lw(\mathcal{K}_\alpha)+2lk(\mathcal{L})$
is used.  Thus, we finish the proof of Theorem 1.4.

\subsection{Case 2: $f_{2-L}(t)=0$}
By formula (4.3), we have
\begin{align}
f_{2-L}(t)=\sum_{l=1}^L\frac{(-1)^{l-1}}{l}\sum_{\cup_{s=1}^l
\mathbf{I}_s=\mathbf{L}}\sum_{r=1}^lh_{2-I_r}^{\mathcal{L}_{\mathbf{I}_r}}(t)\prod_{s\neq
r}h_{-I_s}^{\mathcal{L}_{\mathbf{I}_s}}(t)=0.
\end{align}
Therefore,
\begin{align}
h_{2-L}^{\mathcal{L}}(t)=\sum_{l=2}^L\frac{(-1)^l}{l}\sum_{\cup_{s=1}^l
\mathbf{I}_s=\mathbf{L}}\sum_{r=1}^lh_{2-I_r}^{\mathcal{L}_{\mathbf{I}_r}}(t)\prod_{s\neq
r}h_{-I_s}^{\mathcal{L}_{\mathbf{I}_s}}(t).
\end{align}
one then calculates that, for $L=2$,
\begin{align}
h_0^\mathcal{L}(t)=h_1^{\mathcal{K}_1}(t)h_{-1}^{\mathcal{K}_2}(t)+h_1^{\mathcal{K}_2}(t)h_{-1}^{\mathcal{K}_1}(t).
\end{align}
For $L=3$,
\begin{align}
h_{-1}^\mathcal{L}(t)&=h_0^{\mathcal{L}_{(12)}}(t)h_{-1}^{\mathcal{K}_3}(t)+h_0^{\mathcal{L}_{(13)}}(t)h_{-1}^{\mathcal{K}_2}(t)
+h_0^{\mathcal{L}_{(23)}}(t)h_{-1}^{\mathcal{K}_1}(t)\\\nonumber
&-h_1^{\mathcal{K}_{1}}(t)h_{-1}^{\mathcal{K}_2}(t)h_{-1}^{\mathcal{K}_3}(t)-h_{-1}^{\mathcal{K}_{1}}(t)h_{1}^{\mathcal{K}_2}(t)h_{-1}^{\mathcal{K}_3}(t)
-h_{-1}^{\mathcal{K}_{1}}(t)h_{-1}^{\mathcal{K}_2}(t)h_{1}^{\mathcal{K}_3}(t).
\end{align}
For $L=4$,
\begin{align}
h_{-2}^\mathcal{L}(t)=\sum_{1\leq \beta< \gamma\leq
4}h_{0}^{\mathcal{L}_{(\beta\gamma)}}(t)\prod_{\substack{\alpha=1,..,4
\\ \alpha\neq
\beta,\gamma}}h_{-1}^{\mathcal{K}_\alpha}(t)-2\sum_{\beta=1}^L
h_{1}^{\mathcal{K}_\beta}(t)\prod_{\substack{\alpha=1,..,4\\
\alpha\neq \beta}}h_{-1}^{\mathcal{K}_\alpha}(t).
\end{align}
This suggest that, for the general $L$, the following formula:
\begin{align}
h_{2-L}^\mathcal{L}(t)=\sum_{1\leq \beta< \gamma\leq
L}h_{0}^{\mathcal{L}_{(\beta\gamma)}}(t)\prod_{\substack{\alpha=1,..,L
\\ \alpha\neq
\beta,\gamma}}h_{-1}^{\mathcal{K}_\alpha}(t)-(L-2)\sum_{\beta}h_{1}^{\mathcal{K}_\beta}(t)\prod_{\substack{\alpha=1,..,L\\
\alpha\neq \beta}}h_{-1}^{\mathcal{K}_\alpha}(t).
\end{align}

Next, we will prove it by induction. Assuming it already holds for
the number of components  $\leq L-1$. Then for $L$, by formula
$(4.12)$ and using the induction hypothesis, we have
\begin{align}
h_{2-L}^{\mathcal{L}}(t)&=\sum_{l=2}^{L}\frac{(-1)^l}{l}\sum_{\cup_{s=1}^l
\mathbf{I}_s=\{1,2,..,L\}}\sum_{r=1}^l\left(\sum_{\substack{\beta,\gamma
\in \mathbf{I}_r\\
\beta<\gamma}}h_0^{\mathcal{L}_{(\beta\gamma)}}(t)\prod_{\substack{\alpha\neq
\beta,\gamma\\
\alpha\in
\mathbf{I}_r}}h_{-1}^{\mathcal{K}_\alpha}(t)\right)\prod_{s\neq
r}h_{-I_s}^{\mathcal{L}_{\mathbf{I}_s}}(t)\\\nonumber
&-\sum_{l=2}^{L}\frac{(-1)^l}{l}\sum_{\cup_{s=1}^l
\mathbf{I}_s=\{1,2,..,L\}}\sum_{r=1}^l(I_r-2)\sum_{\beta\in
\mathbf{I}_r}h_1^{\mathcal{K}_\beta}(t)\prod_{\substack{\alpha\neq
\beta\\ \alpha\in
\mathbf{I}_r}}h_{-1}^{\mathcal{K}_\alpha}(t)\prod_{s\neq
r}h_{-I_s}^{\mathcal{L}_{\mathbf{I}_s}}(t)
\end{align}

Let us consider the first summation term on the right side of the
formula (4.17), it is easy to see that this term is symmetric with
respect to the index $\{1,2,..,L\}$. Thus all the items in this
summation have the same coefficient. In order to determine this
coefficient, without loss of generality, we only need to count the
number of terms of the form
\begin{align}
h_{0}^{\caL_{(12)}}(t)h_{-1}^{\caK_3}(t)\cdots h_{-1}^{\caK_{L}}(t)
\end{align}
appearing in the summation
\begin{align}
\sum_{\cup_{s=1}^l
\mathbf{I}_s=\{1,2,..,L\}}\sum_{r=1}^l\left(\sum_{\substack{\beta,\gamma
\in \mathbf{I}_r\\
\beta<\gamma}}h_0^{\mathcal{L}_{(\beta\gamma)}}(t)\prod_{\substack{\alpha\neq
\beta,\gamma\\
\alpha\in
\mathbf{I}_r}}h_{-1}^{\mathcal{K}_\alpha}(t)\right)\prod_{s\neq
r}h_{-I_s}^{\mathcal{L}_{\mathbf{I}_s}}(t)
\end{align}
for a fixed $2\leq l\leq L$. It is easy to see
 this number is, in fact, equal to the number of different
decompositions $\cup_{s=1}^{l}\mathbf{I}_s=\{1,2,\cdot\cdot,L\}$
such that $\caK_1,\caK_2\in \mathbf{I}_i$ for some $1\leq i\leq L$.
We consider two kinds of such
$\{\mathbf{I}_1,\mathbf{I}_2,\cdot,\cdot,\mathbf{I}_l\}$, the first
are those with some $\mathbf{I}_i$, such that $|\mathbf{I}_i|=2$ and
$ \mathbf{I}_i=\{\caK_1,\caK_2\}$, and the second are those with
some $\mathbf{I}_i$, such that $|\mathbf{I}_i|\geq 3,$ and
$\caK_1,\caK_2\in \mathbf{I}_i$. We remark that for $l=L$, this
number is $0$.  Through a straight combinatoric enumeration, we get
the coefficient of
$h_{0}^{\caL_{(12)}}(t)h_{-1}^{\caK_3}(t)\cdot\cdot
h_{-1}^{\caK_{L}}(t)$ as follows
\begin{align}
&\sum_{l\geq 2}^{L-1}\frac{(-1)^l}{l}
\left(l(\sum_{\cup_{i=1}^{l-1}S_{i}=\{1,..,L-2\}}1)+l(\sum_{\cup_{i=1}^{l}S_{i}=\{1,..,L-2\}}1)\right)\\\nonumber
&=\sum_{l\geq 2}^{L-1}(-1)^l
\left(\sum_{\cup_{i=1}^{l-1}S_{i}=\{1,..,L-2\}}1+\sum_{\cup_{i=1}^{l}S_{i}=\{1,..,L-2\}}1\right)
\\\nonumber &=1.
\end{align}
by the formula (5.9) shown in the next section, where the summation
$\sum_{\cup_{i=1}^{l-1}S_{i}=\{1,..,L-2\}}$ will also be explained.
Thus the first summation in (4.17) is simplified to
\begin{align}
\sum_{1\leq \beta< \gamma\leq
L}h_{0}^{\mathcal{L}_{(\beta\gamma)}}(t)\prod_{\alpha\in
\mathbf{L},\alpha\neq \beta,\gamma}h_{-1}^{\mathcal{K}_\alpha}(t).
\end{align}

Now, let us consider the second summation term in the right side of
formula (4.17).  We also note that this term is symmetric with
respect to the index $\{1,..,L\}$. Without loss of generality, we
calculate the coefficient of
\begin{align}
 h_{1}^{\caK_{1}}(t)h_{-1}^{\caK_{2}}(t)\cdots
h_{-1}^{\caK_{L}}(t) \end{align}
in this summation. In fact, for
$2\leq l\leq L$, we only need to count the number decompositions
$\cup_{s=1}^l \mathbf{I}_s=\{1,2,...,L\}$ with $\mathbf{I}_i$  such
that $\caK_1\in \mathbf{I}_i$ and $|\mathbf{I}_i|=k$, for $1\leq
k\leq L-1$. By a straight enumeration, this number is  equal to
\begin{align}
l\cdot\binom{L-1}{k-1}\cdot\left(\sum_{\cup_{i=1}^{l-1}S_i=\{1,2,\cdots,L-k\}}1\right).
\end{align}
Therefore, we can calculate the coefficient of
$h_{1}^{\caK_{1}}(t)h_{-1}^{\caK_{2}}(t)\cdot\cdot
h_{-1}^{\caK_{L}}(t)$ as follows
\begin{align}
&\sum_{l\geq 2}^{L}\frac{(-1)^l}{l} \sum_{k=1}^{L-1}\left((k-2)\cdot
l\cdot\binom{L-1}{k-1}\sum_{\cup_{i=1}^{l-1}S_{i}=\{1,..,L-k\}}1\right)\\\nonumber
&=\sum_{k=1}^{L-1}(k-2)\cdot\binom{L-1}{k-1}\sum_{l\geq
2}^{L}\left((-1)^l\sum_{\cup_{i=1}^{l-1}S_{i}=\{1,..,L-k\}}1\right)\\\nonumber
&=\sum_{k=1}^{L-1}(k-2)\cdot\binom{L-1}{k-1}\sum_{l\geq
2}^{L-k+1}\left((-1)^l\sum_{\cup_{i=1}^{l-1}S_{i}=\{1,..,L-k\}}1\right)
\ \text{by the formula (5.11)}\\\nonumber
&=\sum_{k=1}^{L-1}(k-2)\cdot\binom{L-1}{k-1}(-1)^{L-k+1}\\\nonumber
&=\sum_{k=1}^{L-1}k\cdot\binom{L-1}{k-1}(-1)^{L-k+1}-2\sum_{k=1}^{L-1}\binom{L-1}{k-1}(-1)^{L-k+1}
\ \text{by formula (5.16)}\\\nonumber &=L-2
\end{align}
So the second summation term is simplified  to
\begin{align}
(L-2)\sum_{\beta\in
\mathbf{L}}h_{1}^{\mathcal{K}_\beta}(t)\prod_{\alpha\in \mathbf{L},
\alpha\neq \beta}h_{-1}^{\mathcal{K}_\alpha}(t).
\end{align}
Thus, the induction is completed. We proved the formula (4.16).

By formula (1.5),  we have
$h_{-L}^{\mathcal{L}}(t)=p_{1-L}^{\mathcal{L}}(t)t^{w(\mathcal{L})}(t-t^{-1})$
and
$h_{2-L}^{\mathcal{L}}(t)=p_{3-L}^{\mathcal{L}}(t)t^{w(\mathcal{L})}(t-t^{-1})$.
Substituting them in formula (4.16), Theorem 1.5 is proved.

\section{Some combinatorial identities}
In this section, we will provide the combinatorial formulae used in
 Section 4.

 Let us fix some notation first. A partition $\lambda$ is a finite sequence of positive integers
$(\lambda_1,\lambda_2,..)$ such that $ \lambda_1\geq
\lambda_2\geq\cdots. $ The length of $\lambda$ is the total number
of parts in $\lambda$ and denoted by $\ell(\lambda)$. The degree of
$\lambda$ is defined by $
|\lambda|=\sum_{i=1}^{\ell(\lambda)}\lambda_i. $ If $|\lambda|=d$,
we say $\lambda$ is a partition of $d$ and denoted as $\lambda\vdash
d$. The automorphism group of $\lambda$, denoted by Aut($\lambda$),
contains all the permutations that permute parts of $\lambda$ by
keeping it as a partition. Obviously, Aut($\lambda$) has the order
\begin{align}
|\text{Aut}(\lambda)|=\prod_{i=1}^{\ell(\lambda)}m_i(\lambda)!
\end{align}
where $m_i(\lambda)$ denotes the number of times that $i$ occurs in
$\lambda$. We can also write a partition $\lambda$ as $
\lambda=(1^{m_1(\lambda)}2^{m_2(\lambda)}\cdots). $ For $m\geq 2$,
and $n\leq m$, we will use the notation
$\cup_{i=1}^{n}S_{i}=\{1,..,m\}$ to denote a nonempty disjoint
ordered decomposition of the set $\{1,..,m\}$, i.e. every
$S_i\subset \{1,2,..,m\}$, $S_i\neq \emptyset$ and $S_i\cap
S_j=\emptyset$ for $i\neq j$ such that $S_{1}\cup S_2\cup \cdots
\cup S_n=\{1,..,m\}$, and different orders of $S_{1}, S_2, \cdots,
S_n$ give different decompositions. The summation
$\sum_{\cup_{i=1}^{n}S_{i}=\{1,..,m\}}$ denotes the sum over all
different nonempty disjoint ordered decompositions
$\cup_{i=1}^{n}S_{i}=\{1,..,m\}$.

\begin{lemma}
Suppose $m\geq 2$,
\begin{align}
\sum_{n \geq
2}^{m}\frac{(-1)^{n}}{n}\sum_{\cup_{i=1}^{n}S_{i}=\{1,..,m\}}1=1
\end{align}
\end{lemma}
\begin{proof}
Let $\sum_{i=1}^{n}a_i= m $ and $1\leq a_{n}\leq\cdot\cdot\leq
a_{1}$, through a straight combinatoric enumeration
\begin{align}
\sum_{\cup_{i=1}^{n}S_{i}=\{1,..,m\}}1=\sum_{\{1\leq
a_{n}\leq\cdot\cdot\leq
a_{1}\}}\binom{m}{a_{1},..,a_{n}}\frac{n!}{|Aut((a_{1},..,a_{n}))|}
\end{align}
Then $\{a_{1},\cdot\cdot,a_{n}\}$ denotes a partition $\lambda$ of
$m$ with $n$ components, $\lambda_{i}=a_{i}, \ell(\lambda)=n$. Thus,
the identity we need to prove can be reduced to
\begin{align}
\sum_{\lambda\in S}\frac{(-1)^{\ell(\lambda)}}{\ell(\lambda)}
\frac{\ell(\lambda)!m!}{\lambda_{1}!\cdot\cdot\lambda_{\ell(\lambda)}!|Aut(\lambda)|}=1
\end{align}
where $S=\{\lambda|\lambda\vdash m,\ell(\lambda)\geq 2\}$. Expanding
the right side of the following identity:
\begin{align}
t&=\log(\exp(t))\\\nonumber &=\log(1+\sum_{n\geq
1}\frac{t^n}{n!})\\\nonumber &=\sum_{k \geq
1}^{\infty}\frac{(-1)^{k-1}}{k}\left(\sum_{n\geq
1}\frac{t^n}{n!}\right)^{k}\\\nonumber &=\sum_{k\geq
1}^{\infty}\frac{(-1)^{k-1}}{k}\sum_{\substack{
\{k_{1},k_{2},\cdot,\cdot\cdot\}\\\sum_{i=1}^{n}k_{i}=k
}}\binom{k}{k_{1},k_{2},...}\prod_{i\geq
1}\frac{t^{ik_{i}}}{(i!)^{k_{i}}}
\end{align}
Here we consider a partition $\lambda$ with $k$ components,
$\lambda=(1^{k_1}2^{k_2}\cdot\cdot\cdot)$, $m_{i}(\lambda)=k_{i}$,
let $m=|\lambda|$, then $m=\sum im_{i}(\lambda)=ik_{i}$, we have
$\lambda\vdash m$, $k_{1}+k_{2}+\cdot\cdot=k=\ell(\lambda)$,
$|k_{1}!k_{2}!\cdot\cdot\cdot|=|Aut(\lambda)|$, $\prod_{i\geq
1}(i!)^{k_{i}}=\lambda_{1}!\lambda_{2}!\cdot\cdot\cdot\lambda_{\ell(\lambda)}!$.
Thus the formula (5.5) is
\begin{align}
t=\sum_{m\geq 1}\sum_{\lambda\in
S'}\frac{(-1)^{\ell(\lambda)-1}}{\ell(\lambda)}
\frac{\ell(\lambda)!t^m}{\lambda_{1}!\cdot\cdot\lambda_{\ell(\lambda)}!|Aut(\lambda)|}
\end{align}
where $S'=\{\lambda|\lambda\vdash m, \ell(\lambda)\geq 1\}$.\\
It is clear that the coefficients of $t^{m}$ on the right side are
zero when $m\geq 2$. Equivalently, when $m\geq 2$,
\begin{align}
\sum_{\lambda\in\{\lambda|\lambda\vdash m, \ell(\lambda)\geq 1\}
}\frac{(-1)^{\ell(\lambda)-1}}{\ell(\lambda)}
\frac{\ell(\lambda)!}{\lambda_{1}!\cdot\cdot\lambda_{\ell(\lambda)}!|Aut(\lambda)|}=0
\end{align}
Thus
\begin{align}
\sum_{\lambda\in\{\lambda|\lambda\vdash m, \ell(\lambda)\geq
2\}}\frac{(-1)^{\ell(\lambda)}}{\ell(\lambda)}
\frac{\ell(\lambda)!m!}{\lambda_{1}!\cdot\cdot\lambda_{\ell(\lambda)}!|Aut(\lambda)|}=1
\end{align}
by noting that when $l(\lambda)=1$, $\lambda_1=m$.
\end{proof}

\begin{lemma}
Suppose $m\geq3$,
\begin{align}
\sum_{n\geq 2}^{m}(-1)^n
\left(\sum_{\cup_{i=1}^{n-1}S_{i}=\{1,..,m-1\}}1+\sum_{\cup_{i=1}^{n}S_{i}=\{1,..,m-1\}}1\right)=1
\end{align}
\end{lemma}
\begin{proof}
The formula (5.9) follows from a simple computation as follow
\begin{align}
&\sum_{n\geq 2}^{m}(-1)^n
\left(\sum_{\cup_{i=1}^{n-1}S_{i}=\{1,..,m-1\}}1+\sum_{\cup_{i=1}^{n}S_{i}=\{1,..,m-1\}}1\right)\\\nonumber
&=1+\sum_{k=2}^{m-1}(-1)^k\left(\sum_{\cup_{i=1}^{k}S_{i}=\{1,..,m-1\}}1-\sum_{\cup_{i=1}^{k}S_{i}=\{1,..,m-1\}}1\right)+0\\\nonumber
&=1
\end{align}
\end{proof}

\begin{lemma}
Suppose $m\geq 1$,
\begin{align}
\sum_{n=1}^{m}(-1)^n\sum_{\cup_{i=1}^{n}S_{i}=\{1,..,m\}}1=(-1)^m
\end{align}
\end{lemma}
\begin{proof}
Using the same method as in Lemma 5.1, the identity we
want to prove can be reduced to
\begin{align}
\sum_{\lambda\in\{\lambda|\lambda\vdash m, \ell(\lambda)\geq 1\}
}(-1)^{\ell(\lambda)}\frac{\ell(\lambda)!m!}{\lambda_{1}!\cdot\cdot\lambda_{\ell(\lambda)}!|Aut(\lambda)|}=(-1)^m.
\end{align}
By a straight computation
\begin{eqnarray*}
\exp(-t)&=&\frac{1}{\exp(t)}\\
&=&\frac{1}{1+(\sum_{n\geq 1}^{\infty}\frac{t^n}{n!})}\\
&=&\sum_{k=0}^{\infty}(-1)^k(\sum_{n\geq 1}^{\infty}\frac{t^n}{n!})^k\\
&=& 1+\sum_{k=1}^{\infty}(-1)^k(\sum_{n\geq 1}^{\infty}\frac{t^n}{n!})^k\\
&=&1+\sum_{k=1}^{\infty}(-1)^k
\sum_{\substack{\{k_1,k_2,\cdot\cdot\cdot\}\\ \sum_{i}k_i=k}}
\binom{k}{k_{1},k_{2},...}\prod_{i\geq1}\frac{t^{ik_{i}}}{(i!)^{k_{i}}}\\
&=& 1+\sum_{m\geq 1}\;\sum_{\lambda\in\{\lambda|\lambda\vdash m,
\ell(\lambda)\geq 1\}
}\frac{(-1)^{\ell(\lambda)}\ell(\lambda)!t^m}{\lambda_{1}!\cdot\cdot\lambda_{\ell(\lambda)}!|Aut(\lambda)|}
\end{eqnarray*}
Note that, we also have the expansion
\begin{align}
\exp(-t)=1+\sum_{m\geq 1}\frac{(-1)^m t^m}{m!}
\end{align}
Comparing the coefficients of $t^m$ on both sides,
\begin{align}
\sum_{\lambda\in\{\lambda|\lambda\vdash m, \ell(\lambda)\geq 1\}
}\frac{(-1)^{\ell(\lambda)}\ell(\lambda)!}{\lambda_{1}!\cdot\cdot\lambda_{\ell(\lambda)}!|Aut(\lambda)|}
=\frac{(-1)^m}{m!}
\end{align}
Thus
\begin{align}
\sum_{\lambda\in\{\lambda|\lambda\vdash m, \ell(\lambda)\geq 1\}
}\frac{(-1)^{\ell(\lambda)}\ell(\lambda)!m!}{\lambda_{1}!\cdot\cdot\lambda_{\ell(\lambda)}!|Aut(\lambda)|}=(-1)^m
\end{align}
Hence the identity holds.
\end{proof}

\begin{lemma}
For $n\geq 1$,
\begin{align}
\sum_{k=0}^{n-1}(-1)^{k}(k+1)\binom{n}{k}=(-1)^{n+1}(n+1)
\end{align}
\end{lemma}
\begin{proof}
This is done by straightforward computations
\begin{align}
\sum_{k=0}^{n-1}(-1)^{k}(k+1)\binom{n}{k}&=\sum_{k=0}^{n-1}(-1)^{k}\binom{n}{k}+\sum_{k=1}^{n-1}(-1)^{k}k\binom{n}{k}\\\nonumber
&=\sum_{k=0}^{n-1}(-1)^{k}\binom{n}{k}+\sum_{k=1}^{n-1}(-1)^{k}n\binom{n-1}{k-1}\\\nonumber
&=\sum_{k=0}^{n-1}(-1)^{k}\binom{n}{k}+n\sum_{k=0}^{n-2}(-1)^{k+1}\binom{n-1}{k}\\\nonumber
&=-(-1)^n\binom{n}{n}+n(-1)^{n-1}\binom{n-1}{n-1}\\\nonumber
&=(-1)^{n+1}(n+1)
\end{align}
\end{proof}

\appendix
\renewcommand{\appendixname}{Appendix~\Alph{section}}

\section{The classical approach to the Lickorish-Millett type formulas}
In this section, for the reader's convenience, we give the classical
approach to Theorem 1.4 and Theorem 1.5 as shown in \cite{LM} and
\cite{KM} with a slightly different method from that used in
\cite{CC} to study the Lickorish-Millett type formulas for Kauffman
polynomials. By the formula (1.5), in fact, we only need to prove
the following expressions for $h_{-L}^\mathcal{L}(t)$ and
$h_{2-L}^{\mathcal{L}}(t)$:
\begin{lemma}
For a link $\mathcal{L}$ with $L$ components, the first and second
coefficient polynomials $h_{-L}^\mathcal{L}(t)$ and
$h_{2-L}^{\mathcal{L}}(t)$ are given by
\begin{align}
h_{-L}^{\mathcal{L}}(t)=\prod_{\alpha=1}^{L}h_{-1}^{\mathcal{K}_\alpha}(t).
\end{align}
\begin{align}
h_{2-L}^{\mathcal{L}}(t)=\sum_{1\leq \beta<\gamma \leq
L}h_{0}^{\mathcal{L}_{\beta,\gamma}}(t)\prod_{\alpha\neq
\beta,\gamma}h_{-1}^{\mathcal{K}_\alpha}(t)-(L-2)\sum_{\beta=1}^Lh_{1}^{\mathcal{K}_\beta}(t)\prod_{\alpha\neq
\beta}h_{-1}^{\mathcal{K}_\alpha}(t).
\end{align}
\end{lemma}

Let $\mathcal{L}$ be the link with $L+1$ components. Without loss of
generality, we consider a negative crossing between two components
$\mathcal{K}_1$ and $\mathcal{K}_{L+1}$. Applying the expansion
(1.3) to the skein relation
\begin{align}
\mathcal{H}(\mathcal{L}_+)-\mathcal{H}(\mathcal{L}_-)=z\mathcal{H}(\mathcal{L}_{10L+1}),
\end{align}
where $\mathcal{L}_-=\mathcal{L}$ and $\mathcal{L}_{10L+1}$ denotes
the link $\mathcal{L}_0$ where the crossing belongs to
$\mathcal{K}_1$ and $\mathcal{K}_{L+1}$.

We have
\begin{align}
&h_{-L-1}^{\mathcal{L}_+}(t)z^{-(L+1)}+h_{1-L}^{\mathcal{L}_+}(t)z^{1-L}+h_{3-L}^{\mathcal{L}_+}(t)z^{3-L}+\cdots\\\nonumber
&-(h_{-L-1}^{\mathcal{L}}(t)z^{-(L+1)}+h_{1-L}^{\mathcal{L}}(t)z^{1-L}+h_{3-L}^{\mathcal{L}}(t)z^{3-L}+\cdots)\\\nonumber
&=z(h_{-L}^{\mathcal{L}_{10L+1}}(t)z^{-L}+h_{2-L}^{\mathcal{L}_{10L+1}}(t)z^{2-L}+h_{4-L}^{\mathcal{L}_{10L+1}}(t)z^{4-L}+\cdots)
\end{align}

Comparing the coefficient of $z$,
\begin{align}
h_{-1-L}^{\mathcal{L}_+}(t)&=h_{-1-L}^{\mathcal{L}}(t)\\
h_{1-L}^{\mathcal{L}_+}(t)-h_{1-L}^{\mathcal{L}}(t)&=h_{-L}^{\mathcal{L}_{10L+1}}(t)\\
h_{3-L}^{\mathcal{L}_+}(t)-h_{3-L}^{\mathcal{L}}(t)&=h_{2-L}^{\mathcal{L}_{10L+1}}(t)\\\nonumber
&\cdots
\end{align}
We divide the proof Lemma A.1 into two steps.

{\bf Step 1:} By using the formula (A.3) recursively, we get
\begin{align}
h_{-1-L}^{\mathcal{L}}(t)=h_{-1-L}^{\mathcal{L}_+}(t)=\cdots
=h_{-1-L}^{\mathcal{L}_{1,2,..,L}\otimes \mathcal{K}_{L+1}}(t),
\end{align}
where $\mathcal{L}_{1,2,..,L}\otimes \mathcal{K}_{L+1}$ denotes the
disjoint union of the links $\mathcal{L}_{1,2,..,L}$ and
$\mathcal{K}_{L+1}$.

Since $\mathcal{H}(\mathcal{L}_{1,2,..,L}\otimes
\mathcal{K}_{L+1})=\mathcal{H}(\mathcal{L}_{1,2,..,L})\mathcal{H}(\mathcal{K}_{L+1})$,
 one has
\begin{align}
&h_{-1-L}^{\mathcal{L}_{1,2,..,L}\otimes
\mathcal{K}_{L+1}}(t)z^{-L-1}+h_{1-L}^{\mathcal{L}_{1,2,..,L}\otimes
\mathcal{K}_{L+1}}(t)z^{1-L}+\cdots\\\nonumber &
=(h_{-L}^{\mathcal{L}_{1,2,..,L}}(t)z^{-L}+h_{2-L}^{\mathcal{L}_{1,2,..,L}}(t)z^{2-L}+\cdots)(h_{-1}^{\mathcal{K}_{L+1}}(t)z^{-1}
+h_{1}^{\mathcal{K}_{L+1}}(t)z^{1}+\cdots)
\end{align}

By comparing the coefficients of $z$,
\begin{align}
h_{-1-L}^{\mathcal{L}_{1,2,..,L}\otimes
\mathcal{K}_{L+1}}(t)&=h_{-1}^{\mathcal{K}_{L+1}}(t)h_{-L}^{\mathcal{L}_{1,2,..,L}}(t)\\
h_{1-L}^{\mathcal{L}_{1,2,..,L}\otimes
\mathcal{K}_{L+1}}(t)&=h_{-L}^{\mathcal{L}_{1,2,..,L}}(t)h_{1}^{\mathcal{K}_{L+1}}(t)+h_{-1}^{\mathcal{K}_{L+1}}(t)h_{2-L}^{\mathcal{L}_{1,2,..,L}}(t)
\end{align}

Combing the formulas (A.8) and  (A.10) recursively,
\begin{align}
h_{-1-L}^{\mathcal{L}}(t)= h_{-1-L}^{\mathcal{L}_{1,2,..,L}\otimes
\mathcal{K}_{L+1}}(t)&=h_{-1}^{\mathcal{K}_{L+1}}(t)h_{-L}^{\mathcal{L}_{1,2,..,L}}(t)=\cdots
=\prod_{\alpha=1}^{L+1}h_{-1}^{\mathcal{K}_\alpha}(t).
\end{align}
So we proved the formula (A.1).

{\bf Step 2:} By using the formula  (A.6) and (A.12) together,
\begin{align}
h_{1-L}^{\mathcal{L}}(t)=h_{1-L}^{\mathcal{L}_+}(t)-h_{-L}^{\mathcal{L}_{10L+1}}(t)=h_{1-L}^{\mathcal{L}_+}(t)-h_{-1}^{\mathcal{K}_{10L+1}}
\prod_{\alpha=2}^Lh_{-1}^{\mathcal{K}_\alpha}(t).
\end{align}

Considering the skein relation
\begin{align}
\mathcal{H}((\mathcal{L}_{1,L+1})_+)-\mathcal{H}(\mathcal{L}_{1,L+1})=z\mathcal{H}((\mathcal{K}_{10L+1}).
\end{align}
\begin{align}
&h_{-2}^{(\mathcal{L}_{1,L+1})_+}(t)z^{-2}+h_{0}^{(\mathcal{L}_{1,L+1})_+}(t)z^{0}+\cdots
-(h_{-2}^{\mathcal{L}_{1,L+1}}(t)z^{-2}+h_{0}^{\mathcal{L}_{1,L+1}}(t)z^{0}+\cdots)\\\nonumber
&=z(h_{-1}^{\mathcal{K}_{10L+1}}(t)z^{-1}+h_{1}^{\mathcal{K}_{10L+1}}(t)z+\cdots).
\end{align}
So we have
\begin{align}
h_{0}^{(\mathcal{L}_{1,L+1})_+}(t)-h_{0}^{\mathcal{L}_{1,L+1}}(t)=h_{-1}^{\mathcal{K}_{10L+1}}(t).
\end{align}
Substituting it in (A.13),
\begin{align}
h_{1-L}^{\mathcal{L}}(t)=h_{1-L}^{\mathcal{L}_+}(t)-(h_{0}^{(\mathcal{L}_{1,L+1})_+}(t)-h_{0}^{\mathcal{L}_{1,L+1}}(t))
\prod_{\alpha=2}^Lh_{-1}^{\mathcal{K}_\alpha}(t)
\end{align}
Therefore,
\begin{align}
h_{1-L}^{\mathcal{L}}(t)-h_{0}^{\mathcal{L}_{1,L+1}}(t)
\prod_{\alpha=2}^Lh_{-1}^{\mathcal{K}_\alpha}(t)=h_{1-L}^{\mathcal{L}_+}(t)-h_{0}^{(\mathcal{L}_{1,L+1})_+}(t)
\prod_{\alpha=2}^Lh_{-1}^{\mathcal{K}_\alpha}(t)
\end{align}
and by using it recursively, we obtain
\begin{align}
h_{1-L}^{\mathcal{L}}(t)-h_{0}^{\mathcal{L}_{1,L+1}}(t)
\prod_{\alpha=2}^Lh_{-1}^{\mathcal{K}_\alpha}(t)=h_{1-L}^{\mathcal{L}^{(1)(L+1)}}(t)-h_{0}^{\mathcal{K}_1\otimes
\mathcal{K}_{L+1}}(t)
\prod_{\alpha=2}^Lh_{-1}^{\mathcal{K}_\alpha}(t)
\end{align}
where the notation $\mathcal{L}^{(1)(L+1)}$ denotes the two
components $\mathcal{K}_1$ and $\mathcal{K}_{L+1}$ in $\mathcal{L}$
are unlinked.

Similarly, we apply the above procedure to a crossing between
components $\mathcal{K}_2$ and $\mathcal{K}_{L+1}$ in
$\mathcal{L}^{(1)(L+1)}$, we also have
\begin{align}
h_{1-L}^{\mathcal{L}^{(1)(L+1)}}(t)-h_{0}^{\mathcal{L}_{2,L+1}}(t)
\prod_{\alpha=1,\alpha\neq
2}^Lh_{-1}^{\mathcal{K}_\alpha}(t)=h_{1-L}^{\mathcal{L}^{(12)(L+1)}}(t)-h_{0}^{\mathcal{K}_2\otimes
\mathcal{K}_{L+1}}(t) \prod_{\alpha=1,\alpha\neq
2}^Lh_{-1}^{\mathcal{K}_\alpha}(t),
\end{align}
where the notation $\mathcal{L}^{(12)(L+1)}$ denotes $\mathcal{K}_1,
\mathcal{K}_2$ are unlinked with $\mathcal{K}_{L+1}$ in
$\mathcal{L}$.

Recursively,
\begin{align}
h_{1-L}^{\mathcal{L}}(t)&=h_{1-L}^{\mathcal{L}^{(1)(L+1)}}(t)+h_{0}^{\mathcal{L}_{1,L+1}}(t)
\prod_{\alpha=2}^Lh_{-1}^{\mathcal{K}_\alpha}(t)-h_{0}^{\mathcal{K}_1\otimes
\mathcal{K}_{L+1}}(t)
\prod_{\alpha=2}^Lh_{-1}^{\mathcal{K}_\alpha}(t)\\\nonumber
&=h_{1-L}^{\mathcal{L}^{(12)(L+1)}}(t)+h_{0}^{\mathcal{L}_{2,L+1}}(t)
\prod_{\alpha=1,\alpha\neq
2}^Lh_{-1}^{\mathcal{K}_\alpha}(t)+h_{0}^{\mathcal{L}_{1,L+1}}(t)
\prod_{\alpha=2}^Lh_{-1}^{\mathcal{K}_\alpha}(t)\\\nonumber
&-h_{0}^{\mathcal{K}_1\otimes \mathcal{K}_{L+1}}(t)
\prod_{\alpha=2}^Lh_{-1}^{\mathcal{K}_\alpha}(t)-h_{0}^{\mathcal{K}_2\otimes
\mathcal{K}_{L+1}}(t) \prod_{\alpha=1,\alpha\neq
2}^Lh_{-1}^{\mathcal{K}_\alpha}(t)\\\nonumber &=\cdots\\\nonumber
&=h_{1-L}^{\mathcal{L}_{1,2,..,L}\otimes
\mathcal{K}_{L+1}}(t)+\sum_{\beta=1}^L\left(h_0^{\mathcal{L}_{\beta,L+1}}(t)-h_0^{\mathcal{K}_\beta\otimes
\mathcal{K}_{L+1}}(t)\right)\prod_{\alpha=1,\alpha\neq
\beta}^Lh_{-1}^{\mathcal{K}_\alpha}(t)
\end{align}

Applying the formulas (A.10) and (A.11) to (A.21), we finally obtain
\begin{align}
h_{1-L}^{\mathcal{L}}(t)&=h_{2-L}^{\mathcal{L}_{1,2,..,L}}(t)h_{-1}^{\mathcal{K}_{L+1}}(t)+
\sum_{\beta=1}^Lh_{0}^{\mathcal{L}_{\beta,L+1}}(t)\prod_{\alpha=1,\alpha\neq
\beta}^Lh_{-1}^{\mathcal{K}_\alpha}(t)\\\nonumber
&-(L-1)h_{1}^{\mathcal{K}_{L+1}}(t)\prod_{\alpha=1}^Lh_{-1}^{\mathcal{K}_\alpha}(t)-\sum_{\beta=1}^Lh_{1}^{\mathcal{K}_\beta}(t)
\prod_{\alpha=1,\alpha\neq\beta}^{L+1}h_{-1}^{\mathcal{K}_\alpha}(t).
\end{align}

Let us consider small $L$ cases. For $L=2$, formula (A.22) gives
\begin{align}
h_{-1}^{\mathcal{L}_{1,2,3}}(t)&=h_{0}^{\mathcal{L}_{1,2}}(t)h_{-1}^{\mathcal{K}_3}(t)+h_{0}^{\mathcal{L}_{1,3}}(t)h_{-1}^{\mathcal{K}_2}(t)
+h_{0}^{\mathcal{L}_{2,3}}(t)h_{-1}^{\mathcal{K}_1}(t)\\\nonumber
&-h_{1}^{\mathcal{K}_3}(t)h_{-1}^{\mathcal{K}_1}(t)h_{-1}^{\mathcal{K}_2}(t)-h_{1}^{\mathcal{K}_2}(t)h_{-1}^{\mathcal{K}_3}(t)h_{-1}^{\mathcal{K}_1}(t)
-h_{1}^{\mathcal{K}_1}(t)h_{-1}^{\mathcal{K}_2}(t)h_{-1}^{\mathcal{K}_3}(t)
\end{align}

For $L=3$, we obtain
\begin{align}
h_{-2}^{\mathcal{L}_{1,2,3,4}}(t)&=h_{-1}^{\mathcal{L}_{1,2,3}}(t)h_{-1}^{\mathcal{K}_4}(t)+h_{0}^{\mathcal{L}_{1,4}}(t)h_{-1}^{\mathcal{K}_2}(t)
h_{-1}^{\mathcal{K}_3}(t)\\\nonumber
&+h_{0}^{\mathcal{L}_{2,4}}(t)h_{-1}^{\mathcal{K}_1}(t)
h_{-1}^{\mathcal{K}_3}(t)+h_{0}^{\mathcal{L}_{3,4}}(t)h_{-1}^{\mathcal{K}_2}(t)
h_{-1}^{\mathcal{K}_1}(t)\\\nonumber &-2h_{1}^{\mathcal{K}_{4}}(t)
h_{-1}^{\mathcal{K}_1}(t)h_{-1}^{\mathcal{K}_2}(t)h_{-1}^{\mathcal{K}_3}(t)-h_{1}^{\mathcal{K}_{1}}(t)
h_{-1}^{\mathcal{K}_2}(t)h_{-1}^{\mathcal{K}_3}(t)h_{-1}^{\mathcal{K}_4}(t)\\\nonumber
&-h_{1}^{\mathcal{K}_{2}}(t)
h_{-1}^{\mathcal{K}_1}(t)h_{-1}^{\mathcal{K}_3}(t)h_{-1}^{\mathcal{K}_4}(t)-h_{1}^{\mathcal{K}_{3}}(t)
h_{-1}^{\mathcal{K}_1}(t)h_{-1}^{\mathcal{K}_2}(t)h_{-1}^{\mathcal{K}_4}(t)
\end{align}
Substituting (A.23) in the above formula,
\begin{align}
h_{-2}^{\mathcal{L}_{1,2,3,4}}(t)=\sum_{1\leq \beta<\gamma \leq
4}h_{0}^{\mathcal{L}_{\beta,\gamma}}(t)\prod_{\alpha\neq
\beta,\gamma}h_{-1}^{\mathcal{K}_\alpha}(t)-2\sum_{\beta}h_{1}^{\mathcal{K}_\beta}(t)\prod_{\alpha\neq
\beta}h_{-1}^{\mathcal{K}_\alpha}(t).
\end{align}
From the formulae (A.23) and (A.24) for $L=2,3$,  we guess the
general form as follow:
\begin{align}
h_{2-L}^{\mathcal{L}_{1,2,..,L}}(t)=\sum_{1\leq \beta<\gamma \leq
L}h_{0}^{\mathcal{L}_{\beta,\gamma}}(t)\prod_{\alpha\neq
\beta,\gamma}h_{-1}^{\mathcal{K}_\alpha}(t)-(L-2)\sum_{\beta}h_{1}^{\mathcal{K}_\beta}(t)\prod_{\alpha\neq
\beta}h_{-1}^{\mathcal{K}_\alpha}(t).
\end{align}
We can prove it by induction. In fact, for $L+1$, by the formula
(A.22) and the induction hypothesis,
\begin{align}
h_{1-L}^{\mathcal{L}}(t)&=\sum_{1\leq \beta <\gamma\leq
L}h_{0}^{\mathcal{L}_{\beta,\gamma}}(t)\prod_{\alpha\neq\beta,\gamma}h_{-1}^{\mathcal{K}_\alpha}(t)+\sum_{\beta=1}^L
h_{0}^{\mathcal{L}_{\beta,L+1}}(t)\prod_{\alpha\neq
\beta}h_{-1}^{\mathcal{K}_\alpha}(t)\\\nonumber
&-(L-1)h_{1}^{\mathcal{K}_{L+1}}(t)\prod_{\alpha=1}^Lh_{-1}^{\mathcal{K}_\alpha}(t)-(L-1)\sum_{\beta=1}^Lh_{1}^{\mathcal{K}_\beta}(t)
\prod_{\alpha=1,\alpha\neq
\beta}^{L+1}h_{-1}^{\mathcal{K}_\alpha}(t)\\\nonumber &= \sum_{1\leq
\beta<\gamma \leq
L+1}h_{0}^{\mathcal{L}_{\beta,\gamma}}(t)\prod_{\alpha\neq
\beta,\gamma}h_{-1}^{\mathcal{K}_\alpha}(t)-(L-1)\sum_{\beta=1}^{L+1}h_{1}^{\mathcal{K}_\beta}(t)\prod_{\alpha\neq
\beta}h_{-1}^{\mathcal{K}_\alpha}(t).
\end{align}
Therefore, we complete the proof of the formula (A.2) in Lemma A.1.

\end{document}